\documentclass{amsart}
\usepackage{amssymb}

\newtheorem{Con}{Conjecture}
\newtheorem{Lem}{Lemma}
\newcommand{\Z}{\mathbb{Z}}
\begin{document}
\title{Addendum to "On triangular billard"}
\author{Jan-Christoph Schlage-Puchta}
\maketitle
In \cite{KS}, the authors classified non-isosceles rational-angled triangles with the lattice property subject to the following assumption.
\begin{Con}
Let $1\leq s, t\leq n$ be integers such that $(n,s,t)=1$. Assume that for all $a$ with $(a,n)=1$ we have $\frac{n}{2}<as\bmod n + at\bmod n<\frac{3n}{2}$. Then one of  $n\leq 78$, $s+t=n$, $s+2t=n$, $2s+t=n$, or $n$ even, $|s-t|=\frac{n}{2}$ holds true.
\end{Con}
In \cite{Billardold} the present author proved this statement under the additional assumption that $(n,s)=1$. Due to some misunderstanding it appeared that this was sufficient for the classification of triangles with the lattice property. Recently, Wright pointed out that the proof of the classification is in fact incomplete. Here we remedy this by proving Conjecture 1 in the form given above.

The status of the conjecture before this note is the following.
\begin{Lem}
\label{Lem:previous}
\begin{enumerate}
\item Conjecture 1 holds true if $(n,s)=1$ (cf. \cite{Billardold});
\item Conjecture 1 holds true if $n<10000$ (cf. \cite{KS}).
\end{enumerate}
\end{Lem}

The first result will be used at the end of section~\ref{sec:reduce}, while the second will be used without further mentioning at various places. It turns out that the case $(n,s)=1$ is more difficult than the general case. The reason is that a common divisor introduces some additional symmetry, which greatly simplifies our argument. In particular the bound $n\geq 10000$ is more than sufficient in our argument, whereas in \cite{Billardold} several special cases with $n$ ranging up to $300000$ had to be checked numerically.

\section{Reduction to the case $s=3$}
\label{sec:reduce}
As usual, put $\theta(x, q, a)=\sum\limits_{\underset{p\equiv a\pmod{q}}{p\leq x}}\log p$.
\begin{Lem}
\label{Lem:PNT}
Let $q\leq 10$ be an integer, $(a,q)=1$, and assume that $x\geq 89$. Then $\theta(x,q,a)\geq\frac{x}{2\varphi(q)}$.
\end{Lem}
\begin{proof}
For $x>10^{10}$ this follows from effective estimates for the prime number theorem in arithmetic progressions by Ramar\'e and Rumely, see \cite[Theorem~1]{PNT}. For $x<10^{10}$ we have $\left|\theta(x, q, a)-\frac{x}{\varphi(q)}\right|< 2.072\sqrt{x}$, see \cite[Theorem~2]{PNT}. This implies our claim for $x\geq 619$. Checking the missing range is a trivial task.
\end{proof}
For an integer $n$ denote by $g(n)$ Jacobsthal's function, i.e. the maximal difference between consecutive integers, which are coprime to $n$. 
\begin{Lem}
For $n\not\in\{1,2,3,4,6\}$ we have $g(n)\leq\frac{2}{5}n$.
\end{Lem}
\begin{proof}
Let $q$ be the largest prime power divisor of $n$. Consider the integers $\frac{an}{q}+1$. We have $(\frac{an}{q}+1, n) = (\frac{an}{q}+1, q)$. Moreover we have that $\frac{n}{q}$ is coprime to $q$, thus for any $a$ we obtain that one of $\frac{an}{q}+1$ and $\frac{(a+1)n}{q}+1$ is coprime to $n$. Hence the largest difference between consecutive integers coprime to $n$ is $\leq\frac{2n}{q}$, and our claim follows, provided that $q\geq 5$. If $q<5$, then $n=2^a3^b$ with $a\leq 2$ and $b\leq 1$. For $n=12$ we check directly that $g(12)=4<\frac{24}{5}$, while the other combinations are listed in the Lemma.
\end{proof}
\begin{Lem}
Suppose that $6<(n,t)<\frac{n}{10}$ or $t=5$. Then our claim holds true.
\end{Lem}
\begin{proof}
By multiplying with a suitable $a\in\Z_n^*$ we may assume that $t=(n,t)$. Suppose that $a\equiv 1\pmod{n/t}$ satisfies $(a,n)=1$. Then $a\cdot(\frac{t}{n},\frac{s}{n})\bmod 1 = (\frac{t}{n},\frac{as}{n})$. Write $a=1+k\frac{n}{t}$. If our claim was not true, then the points $\frac{as}{n}\bmod 1=\frac{s}{n}+k\frac{s}{t}\bmod 1$ with $(1+k\frac{n}{t}, n)=1$ do not hit the interval $[0,n/2-t]$. Since $1+k\frac{n}{t}$ is coprime to $\frac{n}{t}$, we have $(1+k\frac{n}{t}, n)=(1+k\frac{n}{t}, t)$. If $k$ ranges over an interval of length $>g(t)$, then this expression has to become 1 for at least one value of $k$ in this interval. Hence for $t>6$ or $t=5$ we obtain that $\{\frac{s}{n}+k\frac{s}{t}\bmod 1:(1+k\frac{n}{t}, n)=1\}$ can only avoid an interval of length $\leq \frac{g(t)}{t}n\leq\frac{2}{5}n$. But we have that this set avoids an interval of length $\frac{n}{2}-t$, thus we obtain $t\geq\frac{n}{10}$.
\end{proof}
\begin{Lem}
Suppose that $\frac{n}{10}\leq(n,t)<\frac{n}{2}$. Then our claim holds true.
\end{Lem}
\begin{proof}
If $(s,n)\geq\frac{n}{10}$, then from $(s,t,n)=1$ we obtain $n\leq 100$. Hence we may assume that $(n,s)\leq 6$. We normalize in such a way that $s=(n,s)$. Put $q=\frac{n}{(n,t)}$, and let $u\in\Z_{n/(n,t)}$ be the solution of $ut\equiv (t,n)\pmod{n}$. Let $a$ be the smallest integer satisfying $(n,a)=1$ and $a\equiv u\pmod{q}$. Then our claim follows, unless $(n,s)a+(n,t)\geq n/2$. If the last condition is not true, then $n$ is divisible by all prime numbers $p\leq\frac{n}{6}$ which satisfy $p\equiv u\pmod{q}$. Hence $\log n\geq\theta(n/6, u, q)$. If we can apply Lemma~\ref{Lem:PNT}, then $\log n\geq\frac{n}{12\varphi(q)}\geq\frac{n}{72}$, thus $n<440$, which is sufficiently small. If on the other hand the lower bound for $\theta$ is not applicable, then $n/6< 89$, which is also sufficiently small.
\end{proof}
Hence we may restrict ourselves to the case $(n,s), (n,t)\in\{1,2,3,4,6\}$. For $(n,s)=1$ the theorem is Lemma~\ref{Lem:previous}, and by symmetry we may assume that  $(n,s), (n,t)\in\{2,3,4,6\}$. But $1=(n,s,t)=((n,s), (n,t))$, thus without loss it suffices to consider the case $(n,s)=3$, $(n,t)\in\{2,4\}$. In particular we have $6|n$. Multyplying by a suitable integer coprime to $n$ we may assume that $s=3$.
\section{The case $s=3$}
For an integer $n$, denote by $f(n)$ the least integer $a$, such that $a(a+2)$ is coprime to $n$.
\begin{Lem}
 Let $p_1, p_2$ be the two largest prime divisors of $n$. If both are greater than 3, then $f(n)\leq\frac{5n}{p_1p_2}$.
 \end{Lem}
\begin{proof}
Let $n'$ be the largest divisor of $n$ coprime to $p_1p_2$. Then $n'\leq\frac{n}{p_1p_2}$. Pick an integer $a$ such that $a(a+2)$ is coprime to $n'$. Since $(n', p_1p_2)=1$ and $p_1, p_2\geq 5$, we have that the five integers $a, a+n', \ldots a+4n'$ are distinct modulo $p_1$ as well as modulo $p_2$. Hence at most two of them are 0 or $-2$ modulo $p_1$, and at most two are 0 or $-2$ modulo $p_2$, hence for some $\nu\in\{0,\ldots, 4\}$ we have that $(a+\nu n')(a+\nu n'+2)$ is coprime to $p_1p_2$. But then $(a+\nu n')(a+\nu n'+2)$ is coprime to $n$, and we conclude that
\[
f(n)\leq a+4n'\leq 5n'\leq\frac{5n}{p_1p_2}.
\]
\end{proof}
\begin{Lem}
Suppose that $n$ is a counterexample to Theorem 1. Then $f(n)\geq\frac{n-12}{18}$.
\end{Lem}
\begin{proof}
We already know that we may assume that $s=3$ and $6|n$. We have $(n, \frac{n}{3}+1)|3$, thus, if $9\nmid n$, then $t, t\cdot(\frac{n}{3}+1), t(\frac{2n}{3}+1)$ are 3 integers, which modulo $n$ are contained in $[\frac{n}{2}-2, n]$. But then $\frac{n}{3}\geq \frac{n}{2}-2$, which is impossible. Now define $d\in\{1,2\}$ by means of the congruence $n\equiv 3d\pmod{9}$. Then $(\frac{dn}{3}+1, n)=1$, and we obtain that both $t$ and $t\cdot(\frac{dn}{3}+1)\bmod n$ are contained in $[\frac{n}{2}-2, n]$. Suppose without loss that $t<t\cdot(\frac{dn}{3}+1)\bmod n$, thus $t\in[\frac{n}{2}-2, \frac{2n}{3}]$. Then we have $\frac{tdn}{3}+t\equiv \frac{n}{3}+t\pmod{n}$, thus $td\equiv 1\pmod{3}$, and therefore $t\equiv\frac{n}{3}\pmod{3}$.

Let $a$ be an integer such that $a(a+2)$ is coprime to $n$. Then we have $a\equiv 2\pmod{3}$, thus $at\equiv -\frac{n}{3}\pmod{3}$, and therefore $at\bmod n\in[\frac{5n}{6}-3a+1, n-1]$. Similarly $(a+2)t\equiv \frac{n}{3}\pmod{3}$, thus $(a+2)t\bmod n\in[\frac{n}{2}-3a-5, \frac{2n}{3}-1]$. Hence
\[
2t\in[\frac{n}{2}-3a-5, \frac{2n}{3}-1]-[\frac{5n}{6}-3a+1, n-1] = [-\frac{n}{6}+3a-2, -\frac{n}{2}-3a-4].
\]
On the other hand
\[
2t\in 2[\frac{n}{2}-2,\frac{2n}{3}-1]=[n-4, \frac{4n}{3}-2].
\]
These two intervals intersect modulo $n$ if and only if either $\frac{n}{6}-3a+2\leq 4$, or $\frac{n}{3}-2\geq\frac{n}{2}-3a-4$. Both conditions are equivalent to $\frac{n}{6}\leq 3a+2$, solving for $a$ yields our claim.
\end{proof}
We can now prove Theorem 1. Suppose $n$ is a counterexample, and let $p_1, p_2$ be the largest prime divisors of $n$. Either at most one of them is $\geq 5$, or we have
\[
\frac{5n}{p_1p_2}\geq f(n)\geq\frac{n-12}{18}>\frac{n}{19},
\]
provided that $n>18\cdot 19$. This implies that the product of the two largest prime divisors of $n$ is $< 95$, thus either $n$ has at most one prime divisor $\geq 5$, or the product of the two largest prime divisors is $<100$. In any case we obtain that $n$ has at most one prime factor $\geq 10$. But then $n$ is coprime to at least one of $11\cdot 13$ or $17\cdot 19$, thus $f(n)\leq 17$ and therefore $n\leq 18\cdot 17+12=318$, which is impossible.

We would like to use this opportunity to correct an error in \cite{Billardold}. In Lemma~4.4 it was falsely claimed that if $g$ is the Jacobsthal function as defined above, $P_k$ is the product of the first $k$ prime numbers, and $n$ is an integer with $k$ distinct prime divisors, then $g(n)\leq g(P_k)$. In general this inequality does not hold, however, in \cite{Billardold} it was only used for $k\leq 8$, where it can be verified directly. In fact, denoting by $\omega(n)$ the number of prime divisors of $n$,  Hajdu and Saradha have computed $\max\{g(n)|\omega(n)=k\}$ for all $k\leq 24$, and they showed that the maximum is attained for $n$ equal to the product of the first $k$ primes, if $k\leq 23$, while for $k=24$ the product of the first 21 primes with 73, 89 and 101 yields a larger value than the product of the first 24 primes. I would like to thank Pasemann for making me aware of this mistake.

\end{document}